\documentclass{article}

\usepackage{color}
\usepackage{amssymb,amsmath}
\usepackage[text={144mm,220mm},centering]{geometry}

\usepackage[numbers]{natbib}
\usepackage{natbibspacing}
\renewcommand{\bibfont}{\small}

\definecolor{myurlcolor}{rgb}{0,0,0.5}
\usepackage{hyperref}
\hypersetup{colorlinks,
linkcolor=black,
citecolor=black,
urlcolor=myurlcolor}

\newcommand{\dashbk}{-}
\newcommand{\such}{\:|\:}
\newcommand{\epsln}{\varepsilon}
\newcommand{\lwr}[1]{\mathbf{#1}}	
\newcommand{\reals}{\mathbb{R}}
\newcommand{\rationals}{\mathbb{Q}}
\newcommand{\demph}[1]{\textbf{\textup{#1}}}
\newcommand{\done}{\hfill\ensuremath{\Box}}
\newenvironment{prooflike}[1]{\begin{trivlist}\item\textbf{#1}\ }
{\end{trivlist}}
\newenvironment{proof}{\begin{prooflike}{Proof}}{\end{prooflike}}
\newcommand{\iso}{\cong}

\newcommand{\toby}[1]{\stackrel{#1}{\to}}
\newcommand{\nm}[1]{\Vert #1 \Vert}
\newcommand{\nrm}[2]{\Vert #1 \Vert_{#2}}
\newcommand{\incl}{\hookrightarrow}
\newcommand{\cln}{\colon}
\newcommand{\Rplus}{\reals_+}
\newcommand{\subprop}[1]{\item\emph{#1:}}
\newcommand{\fprop}[1]{\demph{#1:}}
\newcommand{\hlf}{{\textstyle\frac{1}{2}}}
\newcommand{\MM}[2]{M\bigl( #1,\, #2 \bigr)}
\newcommand{\cooR}{c_{00}^+}
\newcommand{\cooD}{c_{00}^\Delta}
\newcommand{\pos}[1]{\widetilde{#1}}
\newcommand{\dee}{\,d}

\newtheorem{thm}{Theorem}[section]

\newtheorem{lemma}[thm]{Lemma}

\newtheorem{predefn}[thm]{Definition}
\newenvironment{defn}{\begin{predefn}\upshape}{\end{predefn}}
\newtheorem{preexample}[thm]{Example}
\newenvironment{example}{\begin{preexample}\upshape}{\end{preexample}}
\newtheorem{preexamples}[thm]{Examples}

\author{Tom Leinster%
\thanks{School of Mathematics and Statistics, University of Glasgow,
University Gardens, Glasgow
G12 8QW, UK; Tom.Leinster@glasgow.ac.uk.  Supported by an EPSRC Advanced
Research Fellowship.  Mathematics Subject Classification (2010): 26E60
(primary), 47A30, 52A21.}}  
\title{A multiplicative characterization of the power
means} \date{}

\begin{document}

\sloppy
\maketitle

\begin{abstract}
A startlingly simple characterization of the $p$-norms has recently been found
by Aubrun and Nechita~\cite{AuNe} and by Fern\'andez-Gonz\'alez, Palazuelos
and P\'erez-Garc\'{\i}a~\cite{FGPPG}.  We deduce a simple characterization of
the power means of order $\geq 1$.
\end{abstract}

\section*{Introduction}

For each real $p \neq 0$, the power mean (or generalized mean)
of order $p$ assigns the quantity
\[
\Bigl( \frac{1}{n} \sum_{i \in I} x_i^p \Bigr)^{1/p}
\]
to a family $(x_i)_{i \in I}$ of $n$ positive real numbers.  More generally,
uneven weights $(w_i)_{i \in I}$ may be attached to the arguments, giving the
power mean
\[
M_p(w, x) 
= 
\Bigl( \sum_{i \in I} w_i x_i^p \Bigr)^{1/p}.
\]
The power mean of order $0$ is defined as the limit of this expression as $p
\to 0$, namely $M_0(w, x) = \prod_i x_i^{w_i}$.  For the same reason, one
defines $M_{-\infty}(w, x) = \min_i x_i$ and $M_\infty(w, x) = \max_i x_i$.

The basic theory of power means is laid out in the classic text of Hardy,
Littlewood and P\'olya~\cite{HLP}.  In particular, their Theorem~215, when
taken in conjunction with Theorem~84, provides an axiomatic characterization
of the means $M_p$ of order $p \in (0, \infty)$.

Here we give a different characterization, capturing the means $M_p$ of order
$p \in [1, \infty]$.  It is based on the recent characterization by Aubrun and
Nechita~\cite{AuNe} of the $p$-norms
\[
\nrm{x}{p} 
= 
\begin{cases}
\bigl( \sum_{i \in I} |x_i|^p \bigr)^{1/p}      &
\text{if } p < \infty,  \\
\max_{i \in I} |x_i|    &
\text{if } p = \infty.
\end{cases}  
\]
This formula puts a norm on $\reals^I$ for each finite set $I$, and is
\emph{multiplicative}: if $x \in \reals^I$ and $y \in \reals^J$ then $\nrm{x
\otimes y}{p} = \nrm{x}{p} \nrm{y}{p}$, where $x \otimes y \in \reals^I
\otimes \reals^J \iso \reals^{I \times J}$.  Roughly speaking, their
result---which we review below---is that multiplicativity characterizes the
$p$-norms uniquely.  We deduce from it a multiplicative characterization of
the power means.

The theorem proved by Aubrun and Nechita is very closely related to earlier
results of Fern\'andez-Gonz\'alez, Palazuelos and
P\'erez-Garc\'{\i}a~\cite{FGPPG}, although the proofs are not at all similar.
Here it will be more convenient to use Aubrun and Nechita's formulation.

\paragraph*{Acknowledgements} I thank Mark Meckes and Carlos Palazuelos for
useful discussions.

\section{Statement of the theorem}

Write $\Rplus = \{ x \in \reals \such x \geq 0 \}$.   For a finite set $I$,
write 
\[
\Delta_I
=
\bigl\{ 
w \in \Rplus^I \such \sum_{i \in I} w_i = 1
\bigr\}.
\]
For each map $f\cln I \to J$ of finite sets, there is an induced map
$\Delta_I \to \Delta_J$, denoted by $w \mapsto fw$ and defined by $(fw)_j =
\sum_{i \in f^{-1}(j)} w_i$.  There is also an induced map $\Rplus^J \to
\Rplus^I$, denoted by $x \mapsto xf$ and defined by $(xf)_i = x_{f(i)}$.  For
finite sets $I$ and $J$, there are canonical maps
\[
\Delta_I \times \Delta_J \toby{\otimes} \Delta_{I \times J},
\qquad
\Rplus^I \times \Rplus^J \toby{\otimes} \Rplus^{I \times J}
\]
defined by $x \otimes y = (x_i y_j)_{(i, j) \in I \times J}$ whenever $(x, y)
\in \Delta_I \times \Delta_J$ or $(x, y) \in \Rplus^I \times \Rplus^J$.

\begin{defn}
\begin{enumerate}
\item

A \demph{system of means} consists of a function $M: \Delta_I \times \Rplus^I
\to \Rplus$ for each finite set $I$, satisfying:
\begin{description}
\subprop{Functoriality} $M(fw, x) = M(w, xf)$ whenever $f\cln I \to J$ is a
map of finite sets, $w \in \Delta_I$ and $x \in \Rplus^J$.

\subprop{Consistency} $M((1), (c)) = c$ whenever $c \in \Rplus$ and $I$ is a
one-element set, where $(1)$ denotes the unique element of $\Delta_I$ and
$(c)$ is the element of $\Rplus^I$ corresponding to $c$.

\subprop{Monotonicity} $M(w, x) \leq M(w, y)$ whenever $I$ is a finite set, $w
\in \Delta_I$ and $x, y \in \Rplus^I$ with $x_i \leq y_i$ for all $i \in I$.
\end{description}

\item

A system of means $M$ is \demph{convex} if $M(w, \frac{x + y}{2}) \leq \max\{
M(w, x), M(w, y)\}$ whenever $I$ is a finite set, $w \in \Delta_I$ and $x, y
\in \Rplus^I$.

\item 

A system of means $M$ is \demph{multiplicative} if $M(w \otimes v, x \otimes
y) = M(w, x) M(v, y)$ whenever $I$ and $J$ are finite sets, $(w, x) \in
\Delta_I \times \Rplus^I$ and $(v, y) \in \Delta_J \times \Rplus^J$.
\end{enumerate}
\end{defn}

If $M(w, \xi)$ is written as $\int \xi \dee w$, then functoriality becomes the
classical formula for integration under a change of variables or integration
against a push-forward measure: 
\[
\int \xi \dee(f_* w) = \int (\xi\circ f) \dee w.  
\]
(This notation is potentially misleading, since $M(w, \xi)$ need not be linear
in $\xi$.)  The significance of functoriality will be explained further in the
next section.

\begin{example}
For each $p \in [0, \infty]$ there is a multiplicative system of means $M_p$
defined by
\[
M_p(w, x)
=
\begin{cases}
\prod_{i \in I} x_i^{w_i}       &
\text{if } p = 0        \\
\bigl(\sum_{i \in I} w_i x_i^p \bigr)^{1/p}     &
\text{if } 0 < p < \infty       \\
\max_{i \cln w_i > 0} x_i     &
\text{if } p = \infty
\end{cases}
\]
($w \in \Delta_I$, $x \in \Rplus^I$).  If $p \geq 1$ then $M_p$ is convex, by
the triangle inequality for the $p$-norm.  If $p < 1$ then $M_p$ is not
convex, as may be seen by taking 
$w = (1/2, 1/2)$, $x = (1, 0)$ and $y = (0,
1)$.
\end{example}

The purpose of this note is to prove:

\begin{thm}     \label{thm:main}
Every convex multiplicative system of means is equal to $M_p$ for some $p \in
[1, \infty]$.
\end{thm}

\section{Proof of the theorem}
\label{sec:proof}

To prove their characterization theorem for $p$-norms, Aubrun and Nechita use
a standard result of probability theory, Cram\'er's large deviation theorem,
and Fern\'andez-Gonz\'alez, Palazuelos and P\'erez-Garc\'{\i}a use techniques
from the theory of Banach spaces.  In contrast, the deduction of our theorem
from theirs is elementary and almost entirely self-contained.

We begin by recording some elementary properties of systems of means.  We then
take a convex multiplicative system of means, $M$, and extract a number $p \in
[1, \infty]$.  The proof that $M = M_p$ proceeds in two steps.  First we make
the connection between means and norms and apply the $p$-norm characterization
theorem, concluding that $M$ and $M_p$ agree when the weighting is uniform
($w_i = w_j$ for all $i, j$).  Then we apply standard arguments to extend this
result to uneven weightings.

\subsection{Elementary properties of systems of means}

When $I = \{1, \ldots, n\}$ for some integer $n \geq 1$, we write $\Delta_n$
for $\Delta_I$ and $\Rplus^n$ for $\Rplus^I$.  

\begin{lemma}   \label{lemma:fun}
Every system of means $M$ has the following properties.
\begin{enumerate}
\item   \label{prop:sym}
\fprop{Symmetry} 
for all $n \geq 1$, $(w, x) \in \Delta_n \times \Rplus^n$, and permutations
$\sigma \in S_n$,
\[
M(w, x)
=
\MM{(w_{\sigma(1)}, \ldots, w_{\sigma(n)})}%
{(x_{\sigma(1)}, \ldots, x_{\sigma(n)})}.
\]

\item   \label{prop:rep}
\fprop{Repetition} 
for all $n \geq 1$, $w \in \Delta_{n + 1}$, and $x \in \Rplus^n$,
\begin{align*}
&
\MM{(w_1, \ldots, w_{n - 1}, w_n, w_{n + 1})}%
{(x_1, \ldots, x_{n - 1}, x_n, x_n)}\\
=       &
\MM{(w_1, \ldots, w_{n - 1}, w_n + w_{n + 1})}%
{(x_1, \ldots, x_{n - 1}, x_n)}.
\end{align*}

\item   \label{prop:zero}
\fprop{Zero weight}
for all $n \geq 1$, $w \in \Delta_n$, and $x \in \Rplus^{n + 1}$,
\[
\MM{(w_1, \ldots, w_n, 0)}{(x_1, \ldots, x_n, x_{n + 1})}
=
\MM{(w_1, \ldots, w_n)}{(x_1, \ldots, x_n)}.
\]
\end{enumerate}
\end{lemma}

\begin{proof}
Symmetry is proved by applying functoriality to the bijection $\sigma$.
Repetition is proved by applying functoriality to the surjection $\{1, \ldots,
n + 1\} \to \{1, \ldots, n\}$ sending $n + 1$ to $n$ and fixing all other
elements.  The zero weight property is proved by applying functoriality to the
inclusion $\{1, \ldots, n\} \incl \{1, \ldots, n + 1\}$.  \done
\end{proof}

By functoriality applied to bijections, it makes no difference if we restrict
our attention to just one set $\lwr{n} = \{1, \ldots, n\}$ of each
cardinality.  Thus, a system of means may be viewed as a sequence of functions
$(M\cln \Delta_n \times \Rplus^n \to \Rplus)_{n = 1}^\infty$ satisfying
symmetry, repetition, zero weight, consistency, and monotonicity.  To state
the multiplicativity axiom we must choose a bijection $\lwr{m} \times \lwr{n}
\to \lwr{mn}$ for each $m$ and $n$, but by symmetry, the axiom is unaffected
by that choice.

A third option, in the spirit of~\cite{AuNe}, construes a system of means
as a \emph{single} function
\[
M\cln \cooD \times \cooR \to \Rplus, 
\]
where $\cooR$ is the set of finitely-supported sequences in $\Rplus$
and $\cooD = \{ w \in \cooR \such \sum w_i = 1\}$.  It is to satisfy the
evident reformulations of symmetry, repetition, zero weight, consistency, and
monotonicity.  To state the multiplicativity axiom we must choose a
bijection between the set of positive integers and its cartesian square, but
again the choice is immaterial.

The next result says that a weighted mean of numbers increases when
weight is transferred from a smaller number to a larger one.

\begin{lemma}   \label{lemma:transfer}
Every system of means $M$ has the \demph{transfer} property, as follows.  Let
$n \geq 2$, $w \in \Delta_n$, $x \in \Rplus^n$, and $0 \leq \epsln \leq w_n$.
Suppose that $x_n \leq x_{n - 1}$.  Then 
\[
M(w, x) 
\leq 
\MM{(w_1, \ldots, w_{n - 2}, w_{n - 1} + \epsln, w_n - \epsln)}{x}.
\]  
\end{lemma}

\begin{proof}
We have
\begin{align*}
M(w, x) &
=
\MM{(w_1, \ldots, w_{n - 2}, w_{n - 1}, \epsln, w_n  - \epsln)}%
{(x_1, \ldots, x_{n - 2}, x_{n - 1}, x_n, x_n)}      \\
        &
\leq
\MM{(w_1, \ldots, w_{n - 2}, w_{n - 1}, \epsln, w_n  - \epsln)}%
{(x_1, \ldots, x_{n - 2}, x_{n - 1}, x_{n - 1}, x_n)}      \\
        &
=
\MM{(w_1, \ldots, w_{n - 2}, w_{n - 1} + \epsln, w_n - \epsln)}{x}
\end{align*}
by (respectively) repetition, monotonicity, and repetition.  (Symmetry is also
used, but we will generally let this go unmentioned.)  
\done
\end{proof}

\begin{lemma}   \label{lemma:hgs}
Every multiplicative system of means $M$ is \demph{homogeneous}: $M(w, cx) =
cM(w, x)$ whenever $n \geq 1$, $w \in \Delta_n$, $x \in
\Rplus^n$ and $c \in \Rplus$.
\end{lemma}

\begin{proof}
We have
\[
M(w, cx)
=
M((1) \otimes w, (c) \otimes x)
=
M((1), (c)) \cdot M(w, x)
=
cM(w, x)
\]
by definition of $\otimes$, multiplicativity, and consistency.
\done
\end{proof}

\subsection{Recovering the exponent}

For the rest of Section~\ref{sec:proof}, fix a convex multiplicative system of
means $M$.  We will prove that $M = M_p$ for some $p \in [1, \infty]$.

To find $p$, define $\theta: (0, 1) \to \Rplus$ by 
\[
\theta(s) = M((s, 1 - s), (1, 0))
\]
($s \in (0, 1)$).  By multiplicativity and repetition, $\theta(s s') =
\theta(s) \theta(s')$ for all $s, s' \in (0, 1)$, and by transfer, $\theta$ is
(non-strictly) increasing.  If $\theta(s) = 0$ for some $s \in (0, 1)$ then
$\theta(s) = 0$ for all $s \in (0, 1)$.  If not, put $\phi(t) = -\log
\theta(e^{-t})$ ($t > 0$).  Then $\phi$ satisfies the Cauchy functional
equation $\phi(t + t') = \phi(t) + \phi(t')$, and is increasing, from which it
follows easily that $\phi(t) = \alpha t$ for some constant $\alpha \geq 0$ (as
in Section~2.1.1 of~\cite{Acze}).  Hence $\theta(s) = s^\alpha$.  But
\begin{align*}
\theta(\hlf)  &
= 
\max\bigl\{
\MM{(\hlf, \hlf)}{(1, 0)},\:
\MM{(\hlf, \hlf)}{(0, 1)}
\bigr\} \\
        &
\geq
M\bigl((\hlf, \hlf), (\hlf, \hlf)\bigr)
=
M\bigl((1), (\hlf)\bigr)
=
\hlf
\end{align*}
by (respectively) symmetry, convexity, repetition and consistency.  So
$\theta(s) = s^\alpha$ for some $\alpha \in [0, 1]$.  Put $p = 1/\alpha
\in [1, \infty]$: then $\theta(s) = s^{1/p}$.  (In the case $p = \infty$, we
will always understand $1/p$ to mean $0$.)

\subsection{Applying the $p$-norm characterization theorem}

First we recall Theorem~1.1 of Aubrun and Nechita~\cite{AuNe}, rephrasing it
slightly.

Given an injection $f\cln I \to J$ of finite sets, there is an induced map
$\reals^I \to \reals^J$, denoted by $x \mapsto fx$.  It is defined, for $j \in
J$, by $(fx)_j = x_i$ if $j = f(i)$ for some $i \in I$, and $(fx)_j = 0$
otherwise.  Given $x \in \reals^I$ and $y \in \reals^J$, write $x \otimes y =
(x_i y_j)_{(i, j) \in I \times J} \in \reals^{I \times J}$.

A \demph{system of norms} consists of a norm $\nm{\cdot}$ on $\reals^I$ for
each finite set $I$, such that $\nm{fx} = \nm{x}$ whenever $f: I \to J$ is an
injection of finite sets and $x \in \reals^I$.  It is
\demph{multiplicative} if $\nm{x \otimes y} = \nm{x} \nm{y}$ whenever $x \in
\reals^I$ and $y \in \reals^J$.

For example, for each $q \in [1, \infty]$ there is a multiplicative system of
norms $\nrm{\cdot}{q}$ defined by the formula in the Introduction.
Theorem~1.1 of~\cite{AuNe} states that these are the only ones:

\begin{thm}     \label{thm:p-norm}
Every multiplicative system of norms is equal to $\nrm{\cdot}{q}$ for some $q
\in [1, \infty]$.
\end{thm}

We now resume our proof.  For a finite set $I$ with $n \geq 1$ elements,
denote the uniform distribution on $I$ by $u_I = (1/n)_{i \in I} \in
\Delta_I$, and define a function $\nm{\cdot}\cln \reals^I \to \Rplus$ by
\[
\nm{x} 
= 
n^{1/p} 
M\bigl( u_I, (|x_i|)_{i \in I} \bigr)
\]
($x \in \reals^I$).  To cover the case $I = \emptyset$, let $\nm{\cdot}\cln
\reals^\emptyset \to \Rplus$ be the function whose single value is $0$.

When $I = \{1, \ldots, n\}$, we write $u_I$ as $u_n$.  We will use the
observation that
\begin{equation}        \label{eq:prefactor}
n^{-1/p} 
=
M((1/n, 1 - 1/n), (1, 0))
=
M(u_n, (1, 0, \ldots, 0)),
\end{equation}
by the defining property of $p$ and repetition. 

\begin{lemma}
$\nm{\cdot}$ is a norm on $\reals^I$, for each finite set $I$.
\end{lemma}

\begin{proof}
It is enough to prove this when $I = \{1, \ldots, n\}$.  For $x \in \reals^n$
and $c \in \reals$, we have $\nm{cx} = |c| \nm{x}$ by homogeneity.  If $x \in
\reals^n$ with $x \neq 0$ then by symmetry we may assume that $x_1 \neq 0$,
and then
\[
\nm{x} 
\geq
n^{1/p} M(u_n, (|x_1|, 0, \ldots, 0))
=
|x_1| 
> 
0
\]
by monotonicity, homogeneity, and equation~\eqref{eq:prefactor}.

It remains to prove the triangle inequality, or equivalently that the unit
ball $B = \{ x \in \reals^n \such \nm{x} \leq 1 \}$ is convex.  Let $x, y \in
B$ and $\lambda \in [0, 1]$.  We must prove that $\lambda x + (1 - \lambda) y
\in B$.  Write $\pos{x} = (|x_1|, \ldots, |x_n|)$, and similarly $\pos{y}$;
then $\pos{x}, \pos{y} \in B$, and
\[
\nm{\lambda x + (1 - \lambda) y}
\leq
\nm{\lambda \pos{x} + (1 - \lambda) \pos{y}}
\]
by monotonicity, so we may assume that $x_i, y_i \geq 0$ for all $i$.

When $\lambda$ is a dyadic rational, $\lambda x + (1 - \lambda) y \in B$ by
convexity of $M$.  For the general case, let $\epsln > 0$.  Choose a dyadic
$\lambda' \in [0, 1]$ such that $\lambda \leq (1 + \epsln) \lambda'$
and $1 - \lambda \leq (1 + \epsln) (1 - \lambda')$.  Then
\[
\nm{\lambda x + (1 - \lambda) y}        
\leq
\nm{(1 + \epsln) \lambda' x + (1 + \epsln) (1 - \lambda') y}
= 
(1 + \epsln) \nm{\lambda' x + (1 - \lambda') y} 
\leq 
1 + \epsln
\]
by monotonicity and homogeneity.  But $\epsln$ was arbitrary, so $\nm{\lambda
x + (1 - \lambda) y} \leq 1$.
\done
\end{proof}

Next I claim that $\nm{\cdot}$ is a system of norms.  By functoriality, it
suffices to show that
\[
\nm{(x_1, \ldots, x_n)} 
=
\nm{(x_1, \ldots, x_n, 0)}
\]
for all $x \in \reals^n$ ($n \geq 1$).  By definition of $\nm{\cdot}$ and
equation~\eqref{eq:prefactor}, this is equivalent to
\begin{align*}
        &
M(u_{n + 1}, (1, 0, \ldots, 0)) 
\cdot
M(u_n, (|x_1|, \ldots, |x_n|))  \\
=       &
M(u_n, (1, 0, \ldots, 0))
\cdot
M(u_{n + 1}, (|x_1|, \ldots, |x_n|, 0)).
\end{align*}
But by multiplicativity, both sides are equal to
\[
M(u_{n(n + 1)}, (|x_1|, \ldots, |x_n|, 0, \ldots, 0)),
\]
proving the claim.

The system of norms $\nm{\cdot}$ is, moreover, multiplicative, by
multiplicativity of $M$.

Theorem~\ref{thm:p-norm} now implies that $\nm{\cdot} = \nrm{\cdot}{q}$ for
some $q \in [1, \infty]$.  But
\[
\nm{(1, 1)} 
=
2^{1/p} M(u_2, (1, 1))
=
2^{1/p} M((1), (1))
=
2^{1/p}
\]
by definition of $\nm{\cdot}$, repetition and consistency, whereas $\nrm{(1,
1)}{q} = 2^{1/q}$, so $p = q$.  Hence for all $n \geq 1$ and $x \in \Rplus^n$,
\[
M(u_n, x) = n^{-1/p} \nrm{x}{p} = M_p(u_n, x).
\]

\subsection{The case of uneven weights}
\label{subsec:uneven}

We now know that $M(u_n, x) = M_p(u_n, x)$ for all $n \geq 1$ and $x \in
\Rplus^n$.  Here we deduce the same equation for an arbitrary weighting $w$ in
place of $u_n$.  This will complete the proof.

When the coordinates of $w$ are rational, the equation can be proved by a
standard technique, described in Section~2.2 of~\cite{HLP}.  Let $n \geq 1$,
$w \in \Delta_n \cap \rationals^n$ and $x \in \Rplus^n$.  Write $w_i = k_i/k$
where $k_i$ and $k$ are nonnegative integers; thus, $\sum_i k_i = k > 0$.  We
have
\[
M(w, x)
=
M
\Bigl(
\Bigl(
\frac{k_1}{k}, \ldots, \frac{k_n}{k}
\Bigr), x
\Bigr)
=
M\bigl(u_k, 
(
\underbrace{x_1, \ldots, x_1}_{k_1}, 
\ldots, 
\underbrace{x_n, \ldots, x_n}_{k_n})
\bigr)
\]
by repetition and the zero weight property.  (The latter is needed to cover
the possibility that $w_i = 0$ for some values of $i$.)  But $M(u_k, \dashbk)
= M_p(u_k, \dashbk)$, so 
\[
M(w, x) 
=
\Bigl( \sum_{i = 1}^n k_i \frac{1}{k} x_i^p \Bigr)^{1/p}
=
M_p(w, x)
\]
when $p < \infty$, and similarly 
\[
M(w, x) 
= 
\max_{i\cln k_i > 0} x_i 
= 
\max_{i\cln w_i > 0} x_i 
= 
M_\infty(w, x)
\]
when $p = \infty$.

To extend the result to irrational weights, we use an approximation lemma.

\begin{lemma}   \label{lemma:approx}
Let 
$w \in \Delta_n$, $x \in \Rplus^n$, and $\delta > 0$.  Then
there exist $w', w'' \in \Delta_n \cap \rationals^n$ such that 
\[
M(w', x) \geq M(w, x) \geq M(w'', x)
\]
with $\nrm{w - w'}{\infty} < \delta$ and $\nrm{w - w''}{\infty} < \delta$.
\end{lemma}

\begin{proof}
We just prove the existence of such a $w'$, the argument for $w''$ being 
similar.  
Suppose without loss of generality that $x_1 \leq \cdots \leq x_n$.  Choose
$\delta_1 \in [0, \delta)$ with $w_1 - \delta_1 \in \rationals
\cap \Rplus$.  Then
\[
\MM{(w_1 - \delta_1, w_2 + \delta_1, w_3, \ldots, w_n)}{x} 
\geq
M(w, x)
\]
by transfer.  Next choose $\delta_2 \in [0, \delta)$ such that
$w_2 + \delta_1 - \delta_2 \in \rationals \cap \Rplus$.  Then
\[
\MM{(w_1 - \delta_1, w_2 + \delta_1 - \delta_2, w_3 + \delta_2, w_4, 
\ldots, w_n)}{x}        
\geq 
\MM{(w_1 - \delta_1, w_2 + \delta_1, w_3, \ldots, w_n)}{x}. 
\] 
Continuing in this way, we arrive at 
\[
\MM{(w_1 - \delta_1, w_2 + \delta_1 - \delta_2, w_3 + \delta_2 - \delta_3, 
\ldots, w_{n - 1} + \delta_{n - 2} - \delta_{n - 1}, w_n + \delta_{n - 1})}{x}
\geq 
M(w, x).
\]
Taking $w'$ to be the weighting on the left-hand side gives the result.
\done
\end{proof}

Now we prove by induction on $n \geq 1$ that $M(w, x) = M_p(w, x)$ whenever $w
\in \Delta_n$ and $x \in \Rplus^n$.  It is trivial for $n = 1$;
suppose that $n \geq 2$.  Let $w \in \Delta_n$ and $x \in \Rplus^n$.  If $w_i
= 0$ for some $i$ then $M(w, x) = M_p(w, x)$ by the zero weight property and
inductive hypothesis.  Suppose, then, that $w_i > 0$ for all $i$.

Let $\epsln > 0$.  Since $M_p(\dashbk, x)$ is continuous on $\{ w' \in
\Delta_n \such w'_i > 0 \text{ for all } i \}$, we may choose $\delta >
0$ such that $|M_p(w, x) - M_p(w', x)| < \epsln$ for all $w'$ such that
$\nrm{w - w'}{\infty} < \delta$.  Choose $w'$ as in Lemma~\ref{lemma:approx}:
then
\[
M(w, x) 
\leq
M(w', x)
=
M_p(w', x)
<
M_p(w, x) + \epsln.
\]
This holds for all $\epsln > 0$, so $M(w, x) \leq M_p(w, x)$.
Similarly, applying the other half of Lemma~\ref{lemma:approx},
$M(w, x) \geq M_p(w, x)$.  Hence $M(w, x) = M_p(w, x)$,
completing the proof.

\section{A variant}

The zero weight property was used only in Section~\ref{subsec:uneven}, and
only in order to handle means $M(w, x)$ in which $w_i = 0$ for some $i$.  This
suggests a variant of Theorem~\ref{thm:main} in which all weights are
required to be positive and the zero weight property is dropped.  This amounts
to using $\Delta_I^\circ = \{ w \in \Delta_I \such w_i > 0 \text{ for all } i
\in I \}$ in place of $\Delta_I$, and using only \emph{surjections} between
finite sets.

Thus, a \demph{system of positively weighted means} is defined just as a
system of means was defined, but replacing $\Delta_I$ by $\Delta_I^\circ$ and
only demanding functoriality for surjections.  For each $p \in [1, \infty]$
there is a convex multiplicative system of positively weighted means, $M_p$,
defined by restricting the system of means of the same name.  By removing
all mention of zero weights from the proof above, we obtain:

\begin{thm}
Every convex multiplicative system of positively weighted means is equal to
$M_p$ for some $p \in [1, \infty]$.  
\done
\end{thm}

\end{document}